\documentclass[11pt]{amsart}
\pdfoutput=1

\usepackage{amssymb}
\usepackage{mathtools}
\usepackage{microtype}
\usepackage[numbers]{natbib}
\usepackage{hyperref}
\hypersetup{
  pdftitle={Banach space projections and Petrov--Galerkin estimates},
  pdfauthor={Ari Stern},
  pdfsubject={MSC 2010: 65N30, 46B20},
  bookmarksopen=false
}

\theoremstyle{plain} 
\newtheorem{theorem}             {Theorem}[section]
\newtheorem{corollary}  [theorem]{Corollary}

\theoremstyle{definition}
\newtheorem{definition} [theorem]{Definition}

\theoremstyle{remark}
\newtheorem{remark}     [theorem]{Remark}
\newtheorem*{notation}           {Notation}
\newtheorem*{acknowledgment}     {Acknowledgments}

\newcommand{\infsup}[2]{\operatorname{inf\vphantom{sup}}\displaylimits_{#1}
  \operatorname{sup\vphantom{inf}}\displaylimits_{#2}}

\begin{document}
\title{Banach space projections and Petrov--Galerkin estimates}
\author{Ari Stern}
\address{Department of Mathematics\\
Washington University in St.~Louis\\
Campus Box 1146\\
One Brookings Drive\\
St.~Louis, Missouri 63130-4899, USA}
\email{astern@math.wustl.edu}

\begin{abstract}
  We sharpen the classic \textit{a priori} error estimate of
  Babu\v{s}ka for Petrov--Galerkin methods on a Banach space. In
  particular, we do so by (i) introducing a new constant, called the
  \emph{Banach--Mazur constant}, to describe the geometry of a normed
  vector space; (ii) showing that, for a nontrivial projection $P$, it
  is possible to use the Banach--Mazur constant to improve upon the
  na\"ive estimate $ \lVert I - P \rVert \leq 1 + \lVert P \rVert $;
  and (iii) applying that improved estimate to the Petrov--Galerkin
  projection operator. This generalizes and extends a 2003 result of
  Xu and Zikatanov for the special case of Hilbert spaces.
\end{abstract}


\maketitle
 
\section{Introduction}
\label{sec:intro}

In a landmark 1971 paper, \citet{Babuska1971} developed a framework
for the analysis of finite element methods. This analysis encompassed
not only Galerkin methods for coercive bilinear forms (as in the
pioneering work of \citet{Cea1964}), but also Galerkin methods for
non-coercive bilinear forms (such as mixed finite element methods,
cf.~\citet{BrFo1991}) and Petrov--Galerkin methods more generally. A
key innovation in this work was the replacement of the coercivity
assumption by the so-called \emph{inf-sup condition}. (See also the
essential contribution by \citet{Brezzi1974}.) One of the main results
of Babu\v{s}ka's paper is an \textit{a priori} error estimate for
Petrov--Galerkin methods satisfying this inf-sup condition.

Remarkably, more than three decades passed before a 2003 paper, by
\citet{XuZi2003}, pointed out that the constant in Babu\v{s}ka's
estimate can be improved (by $1$) when the space of trial functions is
a Hilbert space. To develop this improved estimate, \citet{XuZi2003}
used an identity concerning the operator norm of a projection on a
Hilbert space. However, this identity is completely idiosyncratic to
Hilbert spaces, and for arbitrary Banach spaces, Babu\v{s}ka's
original estimate has yet to be improved.

The present paper aims to fill this gap, sharpening the constant in
Babu\v{s}ka's estimate for Petrov--Galerkin methods on a Banach
space. The degree of improvement depends on how ``close'' the trial
space is to being Hilbert, in a sense related to Banach--Mazur
distance. In particular, for the most pathological Banach spaces, such
as non-reflexive spaces, no improvement is obtained over Babu\v{s}ka's
estimate, while in the case of Hilbert spaces, we recover Xu and
Zikatanov's improved estimate. The paper is organized as follows:
\begin{itemize}

\item \autoref{sec:background} briefly reviews the results of
  \citet{Cea1964}, \citet{Babuska1971}, and \citet{XuZi2003}, as
  summarized above. In addition to providing the necessary background,
  this also serves to fix the notation and terminology used later in
  the paper.

\item \autoref{sec:banachMazur} introduces the \emph{Banach--Mazur
    constant} of a normed vector space, which quantifies how ``close''
  this space is to being an inner product space. We also show how this
  relates to the well-studied von~Neumann--Jordan constant, which
  serves a similar purpose.

\item \autoref{sec:projection} contains the main technical result: an
  estimate for projection operators on a normed vector space,
  generalizing the Hilbert space projection identity used by {Xu} and
  Zikatanov. This estimate depends fundamentally on the Banach--Mazur
  constant introduced in the previous section.

\item \autoref{sec:lp} illustrates the preceding theory by applying it
  to an important class of Banach spaces: $ L _p $ and Sobolev
  spaces. We compute the Banach--Mazur constants of these spaces and
  discuss the related properties of projection operators, showing that
  the main estimate of \autoref{sec:projection} is sharp.

\item Finally, \autoref{sec:petrovGalerkin} contains the main theorem:
  a sharpened \textit{a priori} error estimate for Petrov--Galerkin
  methods on a Banach space. This is proved by applying the estimate
  from \autoref{sec:projection} to the Petrov--Galerkin projection
  operator.

\end{itemize} 

\begin{acknowledgment}
  Many thanks to Michael Holst and John McCarthy for valuable comments
  and feedback on this work in its early stages.
\end{acknowledgment}

\section{Background: analysis of Petrov--Galerkin methods}
\label{sec:background}

Let $X$ be a Banach space, $Y$ be a reflexive Banach space, and $ a
\in \mathcal{L} ( X \times Y , \mathbb{R} ) $ be a continuous bilinear
form, so that $ \bigl\lvert a ( x, y ) \bigr\rvert \leq M \lVert x
\rVert _X \lVert y \rVert _Y $ for some $ M > 0 $. Given $ f \in Y
^\ast $, we consider the linear problem:
\begin{equation}
  \label{eqn:continuousProblem}
  \text{Find $ u \in X $ such that $ a ( u, v ) = \langle f, v \rangle
    $ for all $ v \in Y $.}
\end{equation} 
If $ X _h \subset X $ and $ Y _h \subset Y $ are closed (e.g.,
finite-dimensional) subspaces, then we also consider the related
problem:
\begin{equation}
  \label{eqn:discreteProblem}
  \text{Find $ u _h \in X _h $ such that $ a ( u _h , v _h ) = \langle
    f, v _h \rangle $ for all $ v _h \in Y _h $.}
\end{equation}
The approximation of \eqref{eqn:continuousProblem} by
\eqref{eqn:discreteProblem} is called the \emph{Petrov--Galerkin
  method}, or simply the \emph{Galerkin method} in the special case $
X _h = Y _h \subset X = Y $.

The most elementary \textit{a priori} error estimate for the Galerkin
method is due to \citet{Cea1964}, who proved that if the bilinear form
satisfies the coercivity condition $ a ( x, x ) \geq m \lVert x \rVert
^2 _X $ for some $ m > 0 $, then
\begin{equation*}
  \lVert u - u _h \rVert _X \leq \frac{ M }{ m } \inf _{ x _h \in X _h
  } \lVert u - x _h \rVert _X .
\end{equation*} 
Note that coercivity is sufficient (although not necessary) for
problems \eqref{eqn:continuousProblem} and \eqref{eqn:discreteProblem}
to be well-posed.  C\'ea's theorem does not apply to the general form
of the Petrov--Galerkin method (since coercivity is meaningless when $
X \neq Y $), nor even to the Galerkin method with non-coercive
bilinear forms, which arise in mixed finite element methods.

A more general condition for \eqref{eqn:continuousProblem} to be
well-posed---which is both necessary and sufficient---is given by the
\emph{inf-sup condition}
\begin{equation*}
  \infsup{ 0 \neq x \in X }{ 0 \neq y \in Y } \frac{ a ( x, y ) }{
    \lVert x \rVert _X \lVert y \rVert _Y } = m > 0 , \qquad \infsup{
    0 \neq y \in Y }{0 \neq x \in X } \frac{ a ( x, y ) }{  \lVert x
    \rVert _X \lVert y \rVert _Y } = m ^\ast > 0.
\end{equation*} 
This is proved by applying Banach's closed range and open mapping
theorems to the operator $ A \colon X \rightarrow Y ^\ast ,\ x \mapsto
a ( x, \cdot ) $, and to its adjoint $ A ^\ast \colon Y \rightarrow X
^\ast ,\ y \mapsto a ( \cdot , y ) $. In fact, when the inf-sup
condition is satisfied, the constants $m$ and $ m ^\ast $ are equal,
since 
\begin{equation*}
 m ^{-1} = \bigl\lVert A ^{-1} \bigr\rVert _{ \mathcal{L} ( Y
  ^\ast , X) } = \bigl\lVert ( A ^\ast ) ^{-1} \bigr\rVert _{
  \mathcal{L} ( X ^\ast , Y ) } = (m ^\ast) ^{-1} .
\end{equation*}
Likewise, the problem \eqref{eqn:discreteProblem} is well-posed if and
only if
\begin{equation*}
  \infsup{ 0 \neq x _h \in X _h }{ 0 \neq y _h \in Y _h } \frac{ a ( x
    _h , y _h  ) }{
    \lVert x _h \rVert _X \lVert y _h \rVert _Y } = \infsup{
    0 \neq y _h \in Y _h }{0 \neq x _h \in X _h  } \frac{ a ( x _h , y
    _h ) }{  \lVert x _h \rVert _X \lVert y _h \rVert _Y } = m _h  > 0
  ,
\end{equation*} 
which is called the \emph{discrete inf-sup condition}.

\citet{Babuska1971} showed that, if the inf-sup conditions are
satisfied, then the solutions to \eqref{eqn:continuousProblem} and
\eqref{eqn:discreteProblem} satisfy the error estimate
\begin{equation*}
  \lVert u - u _h \rVert _X \leq \biggl( 1 + \frac{ M }{ m _h }
  \biggr) \inf _{ x _h \in X _h } \lVert u - x _h \rVert _X .
\end{equation*} 
The proof relies on the \emph{Petrov--Galerkin projection operator} on
$X$, denoted by $ P _h $, which maps each $ u \in X $ to its
Petrov--Galerkin approximation $ u _h \in X _h $. Since $ P _h x _h =
x _h $ for all $ x _h \in X _h $, we have
\begin{equation*}
  \lVert u - u _h \rVert _X = \bigl\lVert ( I - P _h ) u \bigr\rVert
  _X \leq \bigl\lVert ( I - P _h ) ( u - x _h ) \bigr\rVert _X \leq
  \lVert I - P _h \rVert _{ \mathcal{L} ( X , X ) } \lVert u - x _h \rVert _X .
\end{equation*} 
Hence, the estimate follows by observing that
\begin{equation*}
  \lVert I - P _h \rVert _{ \mathcal{L} ( X, X ) } \leq 1 + \lVert P
  _h \rVert _{ \mathcal{L} ( X, X ) } \leq 1 + \frac{ M }{ m _h } ,
\end{equation*} 
and by taking the infimum over all $ x _h \in X _h $.

The Babu\v{s}ka estimate superficially resembles that of C\'ea, with
the glaring exception of $1$ being added to the constant. However,
\citet{XuZi2003} observed that, in the case where $X$ is a Hilbert
space, this additional term is unneccessary, and one obtains the
sharpened estimate
\begin{equation*}
  \lVert u - u _h \rVert _X \leq \frac{ M }{ m _h } \inf _{ x _h \in X
    _h } \lVert u - x _h \rVert _X .
\end{equation*} 
The key insight is that, in a Hilbert space, nontrivial projection
operators $P$ satisfy $ \lVert I - P \rVert _{ \mathcal{L} ( X, X ) }
= \lVert P \rVert _{ \mathcal{L} ( X, X ) } $, so applying this
identity to the Petrov--Galerkin projection yields $ \lVert I - P _h
\rVert _{ \mathcal{L} ( X, X ) } \leq \frac{ M }{ m _h } $.  (See
\citet{Szyld2006} for a discussion of this undeservedly obscure and
frequently rediscovered identity.)

\section{The Banach--Mazur constant of a normed vector space}
\label{sec:banachMazur}

In this section, we introduce the \emph{Banach--Mazur constant} of a
normed vector space $X$, which quantifies the degree to which $X$
fails to be an inner-product space. This constant will play a crucial
role in the projection estimates of \autoref{sec:projection} and
\autoref{sec:petrovGalerkin}. First, we recall the definition of
(multiplicative) Banach--Mazur distance between finite-dimensional
normed vector spaces of equal dimension.

\begin{definition}
  If $V$ and $W$ are finite-dimensional normed vector spaces with $
  \dim V = \dim W $, then the \emph{Banach--Mazur
    distance} between $V$ and $W$ is
  \begin{equation*}
    {d} _{ B M } ( V, W ) = \inf \bigl\{ \lVert T \rVert \lVert T
    ^{-1} \rVert : T \text{ is a linear isomorphism } V \rightarrow W
    \bigr\} .
  \end{equation*} 
\end{definition}

\begin{definition}
  If $X$ is a normed vector space with $ \dim X \geq 2 $, then we
  define the \emph{Banach--Mazur constant} of $X$ to be
  \begin{equation*}
    C _{ B M } (X) = \sup \Bigl\{ \bigl( {d} _{ B M } ( V ,
    \ell _2 ^2 ) \bigr) ^2 : V \subset X,\ \dim V = 2 \Bigr\} ,
  \end{equation*} 
  where $ \ell _2 ^2 $ denotes the two-dimensional $ \ell _2 $ space
  (i.e., $\mathbb{R}^2$ equipped with the Euclidean norm $ \lVert
  \cdot \rVert _2 $).
\end{definition}

\begin{notation}
  For notational brevity, we will omit subscripts from operator norms
  and from $ \lVert \cdot \rVert _X $, denoting each of these simply
  by $ \lVert \cdot \rVert $, where the norm is clear from
  context. The Euclidean norm will always be denoted by $ \lVert \cdot
  \rVert _2 $.
\end{notation}

There are various other such ``geometric constants'' for normed vector
spaces; see \citet{KaTa2010} for a survey of recent results on several
of these constants.  Generally, these constants lie between $1$ and
$2$, equaling $1$ in the case of an inner product space, and equaling
$2$ for the most pathological spaces, such as non-reflexive
spaces. One of the oldest and best-known is the
\emph{von~Neumann--Jordan constant}, dating to the 1935 paper of
\citet{JoVo1935} (see also \citet{Clarkson1937}), which measures the
degree to which the norm satisfies (or fails to satisfy) the
parallelogram law.

\begin{definition}
  The \emph{von Neumann--Jordan constant} of a normed vector space $X$
  is
  \begin{equation*}
    C _{ N J } (X) = \sup \biggl\{ \frac{ \lVert x + y \rVert ^2 +
        \lVert x - y \rVert ^2 }{ 2 \bigl( \lVert x \rVert ^2 + \lVert
        y \rVert ^2 \bigr) } : x, y \in X \text{ not both zero}
      \biggr\} .
  \end{equation*} 
\end{definition}

The following result establishes the relationship between the
Banach--Mazur and von~Neumann--Jordan constants.

\begin{theorem}
  \label{thm:bmInequality}
  $ 1 \leq C _{ N J } (X) \leq C _{ B M } (X) \leq 2 $.
\end{theorem}

\begin{proof}
  The inequality $ 1 \leq C _{ N J } (X) $ appears in
  \citet{JoVo1935}. John's theorem on maximal ellipsoids
  \citep{John1948} implies that $ {d} _{ B M } ( V , \ell_2^2 )
  \leq \sqrt{ 2 } $ for all $ V \subset X $ with $ \dim V = 2 $, and
  thus $ C _{ B M } (X) \leq 2 $. To prove the remaining inequality, $
  C _{ N J } (X) \leq C _{ B M } (X) $, it suffices to show that
  \begin{equation*}
    \lVert x + y \rVert ^2 + \lVert x - y \rVert ^2 \leq 2 C _{ B M }
    (X) \bigl( \lVert x \rVert ^2 + \lVert y \rVert ^2 \bigr) ,
  \end{equation*} 
  for all $ x, y \in X $. This is obvious if $x$ and $y$ are linearly
  dependent, since in that case, they satisfy the parallelogram law
  exactly. Otherwise, take the two-dimensional subspace $ V =
  \operatorname{span} \{ x, y \} $. For any isomorphism $ T \colon V
  \rightarrow \ell_2^2 $,
  \begin{align*}
    \lVert x + y \rVert ^2 + \lVert x - y \rVert ^2 &\leq \lVert T
    ^{-1} \rVert ^2 \bigl( \lVert T x + T y \rVert _2 ^2 + \lVert T x
    - T y \rVert _2 ^2 \bigr) \\
    &= 2 \lVert T ^{-1} \rVert ^2 \bigl( \lVert T x \rVert _2 ^2 +
    \lVert T y \rVert _2 ^2 \bigr) \\
    &\leq 2 \lVert T \rVert ^2 \lVert T ^{-1} \rVert ^2 \bigl( \lVert
    x \rVert ^2 + \lVert y \rVert ^2 \bigr) ,
  \end{align*} 
  where the parallelogram law for $\ell_2^2$ is applied in the
  second line.  Finally, taking the infimum over all $T$ yields
  \begin{align*}
    \lVert x + y \rVert ^2 + \lVert x - y \rVert ^2 &\leq 2 \bigl( {d}
    _{ B M } ( V , \ell_2^2 ) \bigr) ^2 \bigl( \lVert x \rVert ^2
    + \lVert y \rVert ^2 \bigr) \\
    &\leq 2 C _{ B M } (X) \bigl( \lVert x \rVert ^2 + \lVert y \rVert
    ^2 \bigr),
  \end{align*}
  which completes the proof.
\end{proof}

\begin{theorem}
  \label{thm:bmInnerProduct}
  $ C _{ B M } (X) = 1 $ if and only if $X$ is an inner product space.
\end{theorem}

\begin{proof}
  If $X$ is an inner product space, then any two-dimensional subspace
  is unitarily isomorphic to $\ell_2^2$, so $ C _{ B M } (X) = 1
  $. Conversely, if $ C _{ B M } (X) = 1 $, then
  \autoref{thm:bmInequality} implies $ C _{ N J } (X) = 1 $, so by the
  Jordan--von~Neumann theorem \citep{JoVo1935}, $X$ is an inner
  product space.
\end{proof}

\begin{remark}
  The constants $ C _{ B M } (X) $ and $ C _{ N J } (X) $ agree in the
  most extreme cases. For Hilbert spaces, we have seen that $ C _{ B M
  } (X) = C _{ N J } (X) = 1 $. At the opposite extreme, the most
  pathological spaces---including non-reflexive spaces, such as $ L _1
  $ and $ L _\infty $---have $ C _{ N J } (X) = 2 $, and hence $ C _{
    B M } (X) = 2 $ by \autoref{thm:bmInequality} (see also
  \citet{Clarkson1937}). More specifically, a theorem of
  \citet{KaTa1997} states that, if $ C _{ N J } (X) < 2 $, then $X$ is
  super-reflexive. Consequently, if $X$ fails to be super-reflexive
  (in particular, if it is non-reflexive), then $ C _{ N J } (X) = C
  _{ B M } (X) = 2 $.
\end{remark}

\section{A projection estimate for normed vector spaces}
\label{sec:projection}

Having introduced the Banach--Mazur constant, we are now equipped to
prove the main technical result: an estimate for projection operators
on normed vector spaces. This generalizes the Hilbert space projection
identity used by \citet{XuZi2003}.

\begin{theorem}
  \label{thm:projection}
  Let $P$ be a nontrivial projection operator (i.e., $ 0 \neq P = P ^2
  \neq I $) on a normed vector space $X$. Then $ \lVert I - P \rVert
  \leq C \lVert P \rVert $, where $ C = \min \bigl\{ 1 + \lVert P
  \rVert ^{-1} , C _{ B M } (X) \bigr\} $.
\end{theorem}

\begin{proof}
  The inequality $ \lVert I - P \rVert \leq \bigl( 1 + \lVert P \rVert
  ^{-1} \bigr) \lVert P \rVert = 1 + \lVert P \rVert $ is elementary,
  so it suffices to show $ \bigl\lVert ( I - P) x \bigr\rVert \leq C
  _{ B M } (X) \lVert P \rVert \lVert x \rVert $ for all $ x \in X $.

  If $ ( I - P ) x = 0 $, then this inequality is trivial. On the
  other hand, if $ P x = 0 $, then $ ( I - P ) x = x $. Moreover, $
  \lVert P \rVert \geq 1 $ since $P$ is a nontrivial projection, while
  $ C _{ B M } (X) \geq 1 $ by \autoref{thm:bmInequality}. Hence, in
  this case we have $ \bigl\lVert ( I - P ) x \bigr\rVert = \lVert x
  \rVert \leq C _{ B M } (X) \lVert P \rVert \lVert x \rVert $.

  We may now assume that we are in the remaining case, where neither $
  P x $ nor $ ( I - P ) x $ vanishes, so $ V = \operatorname{ span }
  \bigl\{ P x , ( I - P ) x \bigr\} $ is a two-dimensional subspace of
  $X$. If $ T \colon V \rightarrow \ell_2^2 $ is a linear
  isomorphism, then there exist unit vectors $ u, v \in \ell_2^2 $
  and scalars $ a, b \in \mathbb{R} $ such that $ P x = a T ^{-1} u $
  and $ ( I - P ) x = b T ^{-1} v $. Thus, we may write $ x = P x + (
  I - P ) x = a T ^{-1} u + b T ^{-1} v $.

  Now, take $ y = b T ^{-1} u + a T ^{-1} v $, so that $ P y = b T
  ^{-1} u $ and $ ( I - P ) y = a T ^{-1} v $. It follows that
  \begin{align*}
    \bigl\lVert ( I - P ) x \bigr\rVert &= \lVert b T ^{-1} v \rVert\\
    &\leq \lvert b \rvert \lVert T ^{-1} \rVert \\
    &\leq \lVert T \rVert \lVert T ^{-1} \rVert \lVert b T ^{-1} u
    \rVert\\
    &= \lVert T \rVert \lVert T ^{-1} \rVert \lVert P y
    \rVert\\
    &\leq \lVert T \rVert \lVert T ^{-1} \rVert \lVert P \rVert \lVert
    y \rVert .
  \end{align*}
  Next, since $\ell_2^2$ is an inner product space, we have
  \begin{equation*}
    \lVert a u + b v \rVert _2 = ( a ^2 + 2 a b u \cdot v + b ^2 ) ^{ 1/2
    } = \lVert b u + a v \rVert _2 .
  \end{equation*} 
  Therefore,
  \begin{align*}
    \lVert y \rVert &= \lVert b T ^{-1} u + a T ^{-1} v  \rVert\\
    &\leq \lVert T ^{-1} \rVert \lVert b u + a v \rVert _2 \\
    &= \lVert T ^{-1} \rVert \lVert a u + b v \rVert _2 \\
    &\leq \lVert T \rVert \lVert T ^{-1} \rVert \lVert a T ^{-1} u + b
    T ^{-1} v \rVert \\
    &= \lVert T \rVert \lVert T ^{-1} \rVert \lVert x \rVert .
  \end{align*} 
  Altogether, we have now shown that
  \begin{equation*}
    \bigl\lVert ( I - P ) x \bigr\rVert \leq \lVert T \rVert \lVert T
    ^{-1} \rVert \lVert P \rVert \bigl( \lVert T \rVert \lVert T ^{-1}
    \rVert \lVert x \rVert \bigr) = \bigl( \lVert T \rVert \lVert T
    ^{-1} \rVert \bigr) ^2 \lVert P \rVert \lVert x \rVert .
  \end{equation*} 
  Finally, taking the infimum over all isomorphisms $T$ yields
  \begin{equation*}
    \bigl\lVert ( I - P ) x \bigr\rVert \leq \bigl( {d} _{ B M } ( V,
    \ell_2^2  ) \bigr) ^2 \lVert P \rVert \lVert x \rVert \leq C _{ B M }
    (X) \lVert P \rVert \lVert x \rVert ,
  \end{equation*} 
  which completes the proof.
\end{proof}

\begin{corollary}
  If $X$ is an inner product space, then $ \lVert I - P \rVert =
  \lVert P \rVert $.
\end{corollary}

\begin{proof}
  Since $X$ is an inner product space, \autoref{thm:bmInnerProduct}
  implies $ C _{ B M } (X) = 1 $, so \autoref{thm:projection} gives $
  \lVert I - P \rVert \leq \lVert P \rVert $. The reverse inequality
  follows by symmetry of the projections $P$ and $ I - P $.
\end{proof}

\begin{remark}
  \autoref{thm:projection} is strictly sharper than the obvious
  estimate $ \lVert I - P \rVert \leq 1 + \lVert P \rVert $ whenever $
  C _{ B M } (X) < 1 + \lVert P \rVert ^{-1} $. In particular, since $
  1 < 1 + \lVert P \rVert ^{-1} \leq 2 $, this result is always
  sharper when $ C _{ B M } (X) = 1 $ (i.e., for Hilbert spaces) and
  never sharper for the opposite extreme, $ C _{ B M } (X) = 2 $
  (e.g., for non-reflexive spaces). Intermediate cases $ 1 < C _{ B M
  } (X) < 2 $ depend on the particular projection operator $P$.
\end{remark}

\section{Application to $L_p$ and Sobolev spaces}
\label{sec:lp}

In this section, we apply the foregoing theory to $ L _p $ and Sobolev
spaces, which are the most important and commonly-encountered Banach
spaces in finite element analysis.

The simplest possible example is $ X = \ell _p ^2 $, the
two-dimensional $ \ell _p $ space (i.e., $ \mathbb{R}^2 $ equipped
with the $p$-norm), where $ 1 \leq p \leq \infty $. In this case, it
is known that $ {d} _{ B M } ( \ell _p ^2 , \ell _2 ^2 ) = 2 ^{ \lvert
  1/p - 1/2 \rvert } $ (cf.~\citet[Proposition
II.E.8]{Wojtaszczyk1991}), so the Banach--Mazur constant is $ C _{ B M
} (X) = \bigl( {d} _{ B M } ( \ell _p ^2 , \ell _2 ^2 ) \bigr) ^2 = 2
^{ \lvert 2 / p - 1 \rvert } $. If $ 1 \leq p \leq 2 $, consider the
pair of projections
\begin{equation*}
  P ( x _0, x _1 ) = ( x _0 + x _1 , 0 ) , \qquad ( I - P ) ( x _0, x _1 )
  = ( - x _1, x _1 ) .
\end{equation*} 
It can be seen that the operator norms are attained at
\begin{equation*}
  \lVert P \rVert = \frac{ \bigl\lVert P ( 1,1 ) \bigr\rVert _p }{
    \bigl\lVert ( 1 , 1 ) \bigr\rVert _p } = \frac{ \bigl\lVert ( 2,0
    ) \bigr\rVert _p }{ \bigl\lVert (1,1) \bigr\rVert _p } = \frac{ 2
  }{ 2 ^{ 1/p } } = 2 ^{ 1 - 1/p } 
\end{equation*} 
and
\begin{equation*}
  \lVert I - P \rVert = \frac{ \bigl\lVert ( I - P ) ( 0, 1 )
    \bigr\rVert _p }{ \bigl\lVert ( 0 , 1 ) \bigr\rVert _p } =
  \frac{ \bigl\lVert (-1,1) \bigr\rVert _p }{ \bigl\lVert ( 0, 1 )
    \bigr\rVert _p } = \frac{ 2 ^{ 1/p} }{1} = 2 ^{1/p } .
\end{equation*} 
Hence, $ \lVert I - P \rVert = 2 ^{ 2/p -1 } \lVert P \rVert = C _{ B
  M } (X) \lVert P \rVert $; the same can be shown for $ 2 \leq p \leq
\infty $, simply by switching $P$ and $ I - P $. Therefore,
\autoref{thm:projection} is sharp for $ X = \ell _p ^2 $.

More generally, consider $ X = L _p (\mu) $ for some measure $\mu$.
In this case, it is known that $ {d} _{ B M } ( V , \ell _2 ^2 ) \leq
2 ^{ \lvert 1/p - 1/2 \rvert } $ for any two-dimensional subspace $V$
(cf.~\citet[Corollary III.E.9]{Wojtaszczyk1991}). Hence, taking $V$
isometrically isomorphic to $ \ell _p ^2 $---for instance, the span of
two unit-norm functions with disjoint support---implies $ C _{ B M }
(X) = 2 ^{ \lvert 2/p - 1 \rvert } $, as above. In particular, we
obtain the ``best'' case, $ C _{ B M } (X) = 1 $, only for $ p = 2 $;
the ``worst'' case, $ C _{ B M } (X) = 2 $, only for $ p = 1, \infty
$; and the strict inequality $ 1 < C _{ B M } (X) < 2 $ for $ 1 < p <
\infty $.

\begin{remark}
  In fact, here we have $ C _{ B M } (X) = C _{ N J } (X) $, since
  \citet{Clarkson1937} proved that $ C _{ N J } (X) = 2 ^{ \lvert 2 /
    p - 1 \rvert } $ for $ L _p $ spaces.
\end{remark}

Finally, consider the Sobolev space $ X = W ^1 _p (U) $ for $ U
\subset \mathbb{R}^n $. If $ U ^{ \sqcup ( n + 1 ) } $ denotes the
disjoint union of $ n + 1 $ copies of $U$, then we can isometrically
embed $ X \hookrightarrow L _p ( U ^{ \sqcup ( n + 1 ) } ) $ by taking
$ u \mapsto u \oplus \partial _1 u \oplus \cdots \oplus \partial _n u
$. Thus, any two-dimensional subspace of $X$ is isometrically
isomorphic to a two-dimensional subspace of $ L _p ( U ^{ \sqcup (n +
  1) } ) $, and we can again realize $ \ell _p ^2 \subset X $ by
taking the span of two unit-norm functions with disjoint
support. Hence, it follows from the previous discussion that, once
again, $ C _{ B M } (X) = 2 ^{ \lvert 2/p - 1 \rvert } $. More
generally, this argument holds for $ X = W ^k _p (U) $, $ k \in
\mathbb{N} $, since the map $ u \mapsto \bigoplus _{ \lvert \alpha
  \rvert \leq k } \partial _\alpha u $, where $\alpha$ denotes a
multi-index, embeds $X$ isometrically into the space of $ L _p $
functions on sufficiently many disjoint copies of $U$.

\section{The sharpened Petrov--Galerkin estimate}
\label{sec:petrovGalerkin}

We now finally apply \autoref{thm:projection} to the Petrov--Galerkin
projection $ P _h $, using the formalism reviewed in
\autoref{sec:background}.

\begin{theorem}
  \label{thm:petrovGalerkin}
  Let $ u \in X $ and $ u _h \in X _h $ be the solutions to
  \eqref{eqn:continuousProblem} and \eqref{eqn:discreteProblem},
  respectively. As before, let $M , m _h > 0$ denote the continuity
  and discrete inf-sup constants for the bilinear form $ a ( \cdot ,
  \cdot ) $. Then we have the error estimate
  \begin{equation*}
    \lVert u - u _h \rVert \leq C \frac{ M }{ m _h } \inf _{ x _h \in
    X _h } \lVert u - x _h \rVert ,
\end{equation*}
where $ C = \min \bigl\{ 1 + \frac{ m _h }{ M } , C _{ B M } (X)
\bigr\} $.
\end{theorem}

\begin{proof}
  As in Babu\v{s}ka's argument (summarized in
  \autoref{sec:background}), we have
  \begin{equation*}
    \lVert u - u _h \rVert \leq \lVert I - P _h \rVert \inf _{ x _h
      \in X _h } \lVert u - x _h \rVert ,
  \end{equation*} 
  where $ P _h $ is the Petrov--Galerkin projection on $X$. Applying
  \autoref{thm:projection} yields $ \lVert I - P _h \rVert \leq C
  \lVert P _h \rVert \leq C \frac{ M }{ m _h } $, which completes the
  proof.
\end{proof}

\begin{corollary}[\citet{XuZi2003}]
  If $X$ is a Hilbert space, then
  \begin{equation*}
    \lVert u - u _h \rVert \leq \frac{ M }{ m _h } \inf _{ x _h \in X
      _h } \lVert u - x _h \rVert .
  \end{equation*} 
\end{corollary}

\begin{proof}
  By \autoref{thm:bmInnerProduct}, we have $ C = C _{ B M } (X) = 1 $,
  so the result follows immediately from \autoref{thm:petrovGalerkin}.
\end{proof}


\begin{thebibliography}{12}
\providecommand{\natexlab}[1]{#1}

\bibitem[{Babu{\v{s}}ka(1971)}]{Babuska1971}
\textsc{I.~Babu{\v{s}}ka}, \emph{Error-bounds for finite element method},
  Numer. Math., 16 (1971), pp. 322--333.

\bibitem[{Brezzi(1974)}]{Brezzi1974}
\textsc{F.~Brezzi}, \emph{On the existence, uniqueness and approximation of
  saddle-point problems arising from {L}agrangian multipliers}, Rev. Fran\c
  caise Automat. Informat. Recherche Op\'erationnelle S\'er. Rouge, 8 (1974),
  pp. 129--151.

\bibitem[{Brezzi and Fortin(1991)}]{BrFo1991}
\textsc{F.~Brezzi and M.~Fortin}, \emph{Mixed and hybrid finite element
  methods}, vol.~15 of Springer Series in Computational Mathematics,
  Springer-Verlag, New York, 1991.

\bibitem[{C{\'e}a(1964)}]{Cea1964}
\textsc{J.~C{\'e}a}, \emph{Approximation variationnelle des probl\`emes aux
  limites}, Ann. Inst. Fourier (Grenoble), 14 (1964), pp. 345--444.

\bibitem[{Clarkson(1937)}]{Clarkson1937}
\textsc{J.~A. Clarkson}, \emph{The von {N}eumann--{J}ordan constant for the
  {L}ebesgue spaces}, Ann. of Math. (2), 38 (1937), pp. 114--115.

\bibitem[{John(1948)}]{John1948}
\textsc{F.~John}, \emph{Extremum problems with inequalities as subsidiary
  conditions}, in Studies and {E}ssays {P}resented to {R}. {C}ourant on his
  60th {B}irthday, {J}anuary 8, 1948, Interscience Publishers, Inc., New York,
  N. Y., 1948, pp. 187--204.

\bibitem[{Jordan and von Neumann(1935)}]{JoVo1935}
\textsc{P.~Jordan and J.~von Neumann}, \emph{On inner products in linear,
  metric spaces}, Ann. of Math. (2), 36 (1935), pp. 719--723.

\bibitem[{Kato and Takahashi(1997)}]{KaTa1997}
\textsc{M.~Kato and Y.~Takahashi}, \emph{On the von {N}eumann--{J}ordan
  constant for {B}anach spaces}, Proc. Amer. Math. Soc., 125 (1997), pp.
  1055--1062.

\bibitem[{Kato and Takahashi(2010)}]{KaTa2010}
\leavevmode\vrule height 2pt depth -1.6pt width 23pt, \emph{Some recent results
  on geometric constants of {B}anach spaces}, in Numerical Analysis and Applied
  Mathematics (Rhodes, 2010), T.~E. Simos, G.~Psihoyios, and C.~Tsitouras,
  eds., vol. 1281 of AIP Conference Proceedings, 2010, American Institute of
  Physics, pp. 494--497.

\bibitem[{Szyld(2006)}]{Szyld2006}
\textsc{D.~B. Szyld}, \emph{The many proofs of an identity on the norm of
  oblique projections}, Numer. Algorithms, 42 (2006), pp. 309--323.

\bibitem[{Wojtaszczyk(1991)}]{Wojtaszczyk1991}
\textsc{P.~Wojtaszczyk}, \emph{Banach spaces for analysts}, vol.~25 of
  Cambridge Studies in Advanced Mathematics, Cambridge University Press,
  Cambridge, 1991.

\bibitem[{Xu and Zikatanov(2003)}]{XuZi2003}
\textsc{J.~Xu and L.~Zikatanov}, \emph{Some observations on {B}abu\v ska and
  {B}rezzi theories}, Numer. Math., 94 (2003), pp. 195--202.

\end{thebibliography}
\end{document}